\documentclass[a4paper, reqno, 11pt]{amsart}

\usepackage{mlmodern}
\usepackage[centering]{geometry}
\usepackage{enumitem}
\usepackage{amssymb}
\usepackage{hyperref}
\hypersetup{colorlinks=true,
linkcolor=blue,
filecolor=purple, 
urlcolor=cyan,
citecolor=magenta
}

\newtheorem{theorem}{Theorem}[section]
\newtheorem{lemma}[theorem]{Lemma}
\newtheorem{corollary}[theorem]{Corollary}
\newtheorem{proposition}[theorem]{Proposition}
\newtheorem{claim}[theorem]{Claim}
\newtheorem{assumption}{Assumption}

\theoremstyle{definition}
\newtheorem{definition}[theorem]{Definition}
\newtheorem{example}[theorem]{Example}
\theoremstyle{remark}
\newtheorem{remark}[theorem]{Remark}

\allowdisplaybreaks

\numberwithin{equation}{section}

\begin{document}

\title[Roughness of exponential dichotomy in PFDE]{Roughness of exponential dichotomy
\\
under unbounded perturbation in
\\ 
linear partial functional differential equations}
\author[X.-Q. Bui]{Xuan-Quang Bui}
\address{Faculty of Fundamental Sciences, PHENIKAA University, Nguyen Trac Street, Duong Noi Ward, Hanoi 12116, Viet Nam}
\email[corresponding author]{quang.buixuan@phenikaa-uni.edu.vn}

\thanks{The research of X.-Q.~Bui was supported by the Vietnam National Foundation for Science and Technology Development (NAFOSTED) under Grant No. 101.02-2025.56.}
 
\author[N.V. Minh]{Nguyen Van Minh}
\address{Department of Mathematics and Statistics, University of Arkansas at Little Rock, 2801 S University Ave, Little Rock, AR 72204, USA}
\email{mvnguyen1@ualr.edu}
\begin{abstract}
This paper is concerned with the roughness of exponential dichotomies under unbounded perturbations of a class of linear partial functional differential equations
\begin{equation}\label{pfde-000-1star}
u'(t)=Au(t)+Bu_t,
\end{equation}
where $A$ is a linear operator on a Banach space $\mathbb{X}$ and $B$ is a linear operator from $C([-r,0],\mathbb{X})$ into $\mathbb{X}$, where $r>0$ is a given constant. 
To quantify the size of unbounded perturbations, we introduce the \textit{Yosida distance} between linear operators $U$ and $V$, defined by $d_Y(U,V):=\limsup_{\mu\to +\infty} \| U_\mu-V_\mu\|$, where $U_\mu$ and $V_\mu$ are the Yosida approximations of $U$ and $V$, respectively. 
We show that if $d_Y(A, A_1)$ and $d_Y(B, B_1)$ are sufficiently small, then the perturbed equation
\begin{equation}\label{pfde-000-2star}
u'(t)=A_1u(t)+B_1u_t
\end{equation}
also admits an exponential dichotomy whenever \eqref{pfde-000-1star} admits one. 
The proofs are based on estimates of the Yosida distance between the generators of the solution semigroups associated with \eqref{pfde-000-1star} and \eqref{pfde-000-2star} in the phase space $C([-r,0],\mathbb{X})$, without assuming any relation between their domains.
\end{abstract}
\date{\today}
\maketitle
{
\parskip 0pt
\tableofcontents
}

\section{Introduction}
In this paper, we investigate unbounded perturbations of delayed evolution equations, which serve as abstract models of partial functional differential equations
\begin{equation}\label{PFDE-00}
u'(t)=Au(t)+Bu_t, \quad t \geq 0,
\end{equation}
where $A$ generates a $\mathrm{C}_0$-semigroup in a Banach space $\mathbb{X}$ and $B\in \mathcal{L}(\mathcal{C} ,\mathbb{X})$, where $\mathcal{C} :=C([-r,0],\mathbb{X})$, $r>0$ is a given positive constant.
For $u(\cdot )$ is a function from $[-r,\epsilon)$ to $\mathbb{X}$ with $\epsilon >0$, then for $t\in [0,\epsilon)$, we denote by $u_t\in \mathcal{C}$ the function $u_t(\theta ) :=u(t+\theta )$, $ \theta \in [-r,0]$. As is well known, for reasonable conditions on $A$ and $B$ (see \cite{traweb}) Eq.~\eqref{PFDE-00} generates a $\mathrm{C}_0$-semigroup in $\mathcal{C}$ referred to as the "solution semigroup". 
We refer the reader to the monographs in \cite{hal, wu} for a comprehensive account of the theory and applications of functional differential equations to real-world problems (see also \cite{batpia, diegilverwal, driessezz}). 
The interest in functional differential equations has been growing recently as they are better models for real world processes that are determined by delay factors.

It is well known in the theory of dynamical systems that for a linear system of the form
\[
x'(t)=Ax(t),\quad x(t)\in \mathbb{R}^n
\]
under "small perturbation" some asymptotic behavior like exponential dichotomy persists (see e.g. \cite{cop, dalkre, engnag}, see also \cite{huy2006, huymin, minrabsch}), that is, the perturbed system 
\[
x'(t)=(A+B)x(t)
\]
also has an exponential dichotomy if $B$ is small. In the finite-dimensional case the size of a small perturbation can be determined  by the norm $\| B\|$. 
Many results from finite-dimensional dynamical systems could be extended to the infinite-dimensional case. In this paper, assuming that the evolution equation
\begin{equation}\label{nd}
u'(t)=Au(t),\quad u(t)\in \mathbb{X},
\end{equation}
where $\mathbb{X}$ is a Banach space, generates a $\mathrm{C}_0$-semigroup in $\mathbb{X}$ that has an exponential dichotomy, we consider the perturbed equation
\[
u'(t)=(A+B)u(t),
\]
where $B$ is (generally unbounded) linear operator on $\mathbb{X}$. 

A key difficulty in studying this equation is related to understanding how to measure the size of the perturbation $B$. It is naturally easy to extend the persistence of exponential dichotomy under small perturbation $B$ in the sense of its norm, that is, $\| B\|$ is sufficiently small. A more general class of unbounded perturbation $B$ was discussed in \cite{cholei, pru} (see also \cite[p. 630--631]{dunsch} and \cite{pralun, yangimllv, zholuzha}).

As for delayed equation Eq.~\eqref{PFDE-00}, under reasonable conditions on $A$ and $B$, it generates a $\mathrm{C}_0$-semigroup (called solution semigroup) in the phase space $\mathcal{C}$, so its asymptotic behavior is studied via this semigroup. It is expected that many results of the non-delayed equation \eqref{nd} could be easily extended to the delayed equations \eqref{PFDE-00}. That means, if the unperturbed equation \eqref{PFDE-00} has an exponential dichotomy, then under "small perturbation" $A_1$ and $B_1$, the perturbed equation
\begin{equation}\label{pe}
u'(t)=(A+A_1)u(t)+(B+B_1)u_t 
\end{equation}
also has an exponential dichotomy. Even when $A_1=0$ and $B_1$ is bounded and $\| B_1\|$ is small, the study of the persistence of exponential dichotomy, to the best of our knowledge, has yet to be complete. In fact, the problem is that the Variation-of-Constants Formula in the phase space $\mathcal{C}$ is not available for the solution semigroups associated with Eq.\eqref{PFDE-00} and Eq.~\eqref{pe}. 
It turns out that in the infinite-dimensional case, the familiar form of Variation-of-Constants Formula in the phase space $\mathcal{C}$ is not valid as shown in \cite{hinmurnaimin}. Instead, a weaker version of Variation-of-Constants Formula was proposed. With a weaker Variation-of-Constants Formula, as shown in \cite{adiezzouh, adiezzouh2, adiezzlak, murmin},  some bounded (or nonlinear Lipchitz) perturbation $\|B_1\|$ of the operator $B$ could be studied.

To our knowledge, the persistence of exponential dichotomy of Eq.~\eqref{pe} under unbounded perturbations $A_1$ and $B_1$ has not been studied yet. The idea of applying the method by \cite{cholei} or  \cite[p.~630--631]{dunsch} to study the persistence of exponential dichotomy of the solution semigroup generated by Eq.\eqref{PFDE-00} in the phase space $\mathcal{C}$ is impossible. In fact, the very first condition on the domain for the generator of solution semigroups is not satisfied for delayed equations as shown in our Corollary~\ref{cor 100} and Example~\ref{exa 1}.

In our paper we aim to study the persistence of exponential dichotomy of the solution semigroup generated by Eq.~\eqref{PFDE-00} where both $A$ and/or $B$ are perturbed by unbounded operators $A_1$ and $B_1$ from some classes. We will approach this problem via the concept of Yosida distance as defined below. The idea of using the Yosida distance arises from the Theory of Semigroups of Bounded Operators, where Yosida approximation $A_\lambda$ is used to construct the semigroup generated by a linear equation $u'(t)=Au(t)$. The advantage of using Yosida distance between two unbounded operators is that we can free the domain conditions. Our Lemma~\ref{lem 1} below shows that this concept is a natural extension of the concept of distance between bounded operators to some common classes of unbounded operators, so it is suitable to measure the size of perturbation even in the unbounded case.

Our paper is organized as follows: After Section~\ref{Sect-Pr} where we list some notations to be used in the paper as well as some known results, we present the main results of the paper in Sections~\ref{Sect-YdPe}, \ref{Sect-PePFDE} and \ref{Sect-Pealphanorm}. 
In the last section (Section~\ref{Sect-examples}) we give certain examples illustrating how the paper's results could be applied to partial functional differential equations under unbounded perturbation. 
Our key concept of Yosida distance is given in Definition~\ref{def 1} and Theorem~\ref{the per} on persistence under small perturbation in the sense of Yosida distance. In Lemma~\ref{lem 1} we consider several commonly encountered classes of linear operators and their perturbations. The perturbed equations Eq.~\eqref{pe} with unbounded $A_1$ and bounded $B_1$ are studied in Section~\ref{Sect-PePFDE}, while the case of both unbounded $A_1$ and $B_1$ is studied in Section~\ref{Sect-Pealphanorm}. Example~\ref{Exa1} shows that the domain condition on the perturbation operators is not satisfied in the case of the solution semigroups generated by partial functional differential equations. 
This is a justification for a different approach, the Yosida distance approach, to the perturbation theory of partial functional differential equations is needed. 
The estimates of Yosida distance of the generators of solution semigroups generated by the unperturbed and perturbed equations are given in Lemma~\ref{lem 39} and Lemma~\ref{lem 40} that are the key for us to apply Theorem~\ref{the per}. 
Our main results are stated in Theorems~\ref{the main1} and \ref{the main2}. 
To the best of our knowledge, these results are new. 
Preliminary versions of some results appeared in \cite{buimin}.

Before concluding this section, we would like to make a brief remark on the notion of the so-called \emph{Yosida distance}, which plays a central role throughout this work. 
The concept of the \emph{Yosida distance} was first introduced in \cite{buimin} with the original purpose of studying the roughness of exponential dichotomies and the existence of invariant manifolds for nonlinear evolution equations. 
Since then, this concept has proved to be a useful tool for investigating several problems in the qualitative theory of differential equations and dynamical systems. 
For instance, it provides a natural measure in the study of the stability radius and the dichotomy radius of infinite-dimensional linear systems under unbounded perturbations, see \cite{buiminthu}. 
It also leads to a generation theorem for unbounded autonomous perturbations of $\mathrm{C}_0$-semigroups in \cite{buihuyluomin}, as well as results on the well-posedness of nonautonomous perturbations of $\mathrm{C}_0$-semigroups in \cite{builuomin}. 
As shown in the present work, the Yosida distance can also be applied to the study of unbounded perturbations of linear partial functional differential equations. 
These results confirm the suitability and usefulness of the Yosida distance. Along with previous studies, it emerges as a new tool in the analysis of unbounded perturbations of evolution equations.

\section{Preliminaries}\label{Sect-Pr}

\subsection{Notations}
Throughout this paper, $\mathbb{X}$ and $\mathbb{Y}$ denote Banach spaces with their respective norms. The symbols $\mathbb{R}$ and $\mathbb{C}$ stand for the fields of real and complex numbers, respectively. 
We denote by $\mathcal{L}(\mathbb{X},\mathbb{Y})$ the Banach space of all bounded linear operators from $\mathbb{X}$ to $\mathbb{Y}$. For a linear operator $T$, its domain and range are denoted by $D(T)$ and $R(T)$, respectively. The resolvent set and spectrum of $T$ are denoted by $\rho(T)$ and $\sigma(T)$. 
For $\lambda \in \rho(T)$, we write
$
R(\lambda,T) := (\lambda I - T)^{-1}
$
for the resolvent of $T$ at $\lambda$.  
For a linear operator $P$ on $\mathbb{X}$, we use $\ker(P)$ and $\mathrm{Im}(P)$ to denote its kernel and image, respectively. 
For a complex number $z$, its real and imaginary parts are denoted by $\Re z$ and $\Im z$, respectively.

\subsection{Exponential Dichotomy}
In this subsection, we present some preliminary material on exponential dichotomies. In particular, we recall the relevant definitions and known results, and conclude with a characterization of exponential dichotomy in terms of spectral properties.
\begin{definition}[Exponential dichotomy]
A $C_0$-semigroup $\left (T(t)\right )_{t \ge 0}$ on a Banach space $\mathbb{X}$ is said to admit an \textit{exponential dichotomy} (or to be \textit{hyperbolic}) if there exist a bounded projection $P$ on $\mathbb{X}$ and constants $N \ge 1$, $\alpha > 0$ such that
\begin{enumerate}
\item $T(t) P = P T(t)$, for $t \ge 0$;
\item for each $t \ge 0$, the restriction $T(t)\big|_{\ker(P)}$ is an isomorphism from $\ker(P)$ onto itself, and we define
\[
T(-t) := \left (T(t)\big|_{\ker(P)}\right )^{-1};
\]
\item the following estimates hold 
\begin{align*}
&
\|T(t)x\| \le N e^{-\alpha t} \|x\|,
\quad \text{ for all }
t \ge 0,\quad
x \in \mathrm{Im}(P),
\\
&
\|T(-t) x\| \le N e^{-\alpha t} \|x\|,
\quad \text{ for all }
t \ge 0,\quad
x \in \ker (P).
\end{align*}
\end{enumerate}
\end{definition}

The projection $P$ is called the \textit{dichotomy projection} for the hyperbolic semigroup $\left (T(t)\right )_{t \ge 0}$, and the constants $N$ and $\alpha$ are called \textit{dichotomy constants}.

The following characterization is well known:
\begin{lemma}\label{Lemhyp}
Let $\left (T(t)\right )_{t \ge 0}$ be a $\mathrm{C}_0$-semigroup. Then $\left (T(t)\right )_{t \ge 0}$ admits an exponential dichotomy if and only if $\sigma (T(1)) \cap \{ z\in \mathbb{C}: |z|=1\}=\emptyset$.
\end{lemma}
\begin{proof}
See \cite[1.17 Theorem]{engnag}.
\end{proof}

As a consequence of Lemma~\ref{Lemhyp}, the following result holds.
\begin{lemma}
Let $\left (T(t)\right )_{t \ge 0}$ be a $\mathrm{C}_0$-semigroup that admits an exponential dichotomy. Then $\left (S(t)\right )_{t \ge 0}$ admits an exponential dichotomy provided that $S(1)$ is sufficiently close to $T(1)$.
\end{lemma}
\subsection{Fractional Powers of Closed Operators}
In this subsection, we present some results on fractional powers of closed linear operators.

\begin{definition}[Sectorial operator, see {\cite[Definition 1.3.1]{hen}}]
A linear operator $A$ on a Banach space $\mathbb{X}$ is said to be \textit{sectorial} if it is closed and densely defined, and there exist constants $\phi \in \left(0,\frac{\pi}{2}\right)$, $M \ge 1$, and $a \in \mathbb{R}$ such that the sector
\[
S_{a, \phi} := 
\left \{
\lambda : \phi \leq |\arg (\lambda - a)| \leq \pi,\, \lambda \neq a
\right \}
\]
is contained in the resolvent set of $A$ and
\[
\left \| (\lambda - A)^{-1}\right \| \leq \frac{M}{|\lambda - a|},
\quad \text{ for all } \lambda \in S_{a, \phi}.
\]
\end{definition}

If $-A$ is the generator of an analytic semigroup we will use the fractional powers as defined in \cite[1.4, p. 24]{hen} and \cite[2.6]{paz}. To be specific, we assume that $0\le \alpha \le 1$ we will denote $\mathbb{X}^\alpha := D(A^\alpha_1)$, where $A_1:=A+aI$ with $a$ chosen so that $\Re \sigma (A_1)>0$ with  the graph norm $\| x\|_\alpha$ defined as
\[
\| x\|_\alpha := \| A^\alpha _1 x\|,\quad x\in \mathbb{X}^\alpha.
\]
Since different $a$ will determine equivalent different norms we can fix a number $a$. 
As is well known (see \cite[Theorem 1.4.8]{hen}) if $-A$ is the generator of an analytic semigroup, then $\mathbb{X}^\alpha$ with the norm defined above is a Banach space. 
Moreover, by \cite[Theorem 1.4.3]{hen}, \cite[Theorem 6.13, p. 74]{paz}, we have
\begin{theorem}
Let $-A$ be the infinitesimal generator of an analytic semigroup $\left (T(t)\right )_{t\ge 0}$. If $\Re \sigma (A) > \delta$, then, 
\begin{enumerate}
\item $T(t):\mathbb{X} \to D(A^\alpha)$ for every $t>0$ and $\alpha \ge 0$.
\item For every $x\in D(A^\alpha )$ we have $T(t)A^\alpha x=A^\alpha T(t)x$;
\item For every $t>0$ the operator $A^\alpha T(t)$ is bounded and
\[
\| A^\alpha T(t)\| \le M_\alpha t^{-\alpha} e^{-\delta t}.
\]
\item Let $0<\alpha \le 1$ and $x\in D(A^\alpha )$, then
\[
\| T(t)x-x\| \le C_\alpha t^\alpha \| A^\alpha x\| .
\]
\end{enumerate}
\end{theorem}

\begin{theorem}[{see \cite[Theorem 1.4.8]{hen}}]
\label{ThmHenryFower}
If $A$ is sectorial in a Banach space $\mathbb{X}$, then $\mathbb{X}^{\alpha}$ is a Banach space in the norm $\|\cdot\|_{\alpha}$, for $\alpha \geq 0$, and for $\alpha \geq \beta \geq 0$, $\mathbb{X}^{\alpha}$ is a dense subspace of $\mathbb{X}^{\beta}$ with continuous inclusion. 
If $A$ has compact resolvent, the inclusion $\mathbb{X}^{\alpha} \subset \mathbb{X}^{\beta}$ is compact when $\alpha > \beta \geq 0$.

If $A_1$ and $A_2$ are sectorial operators in $\mathbb{X}$ with the same domain and $\Re \sigma (A_j) > 0$, for $i \in \{1, 2\}$, and if $(A_1-A_2)A_1^{-\alpha}$ is a bounded operator for some $\alpha < 1$, then with $\mathbb{X}_j^{\beta}=D(A_j^{\beta})$, $\mathbb{X}_1^{\beta} = \mathbb{X}_2^{\beta} $ with equivalent norms for $0 \leq \beta \leq 1$.
\end{theorem}

\subsection{Partial Functional Differential Equations}
Let us consider partial functional differential equations of the form
\[
u'(t)=Au(t)+Bu_t, \quad t \geq 0,
\]
where $A$ generates a $\mathrm{C}_0$-semigroup in a Banach space $\mathbb{X}$ and $B\in \mathcal{L}(\mathcal{C} , \mathbb{X})$. 
In this paper we will use the standard theory on partial functional differential equations developed in \cite{traweb} for Eq.~\eqref{PFDE-00}. For simplicity, we may assume that the $\mathrm{C}_0$-semigroup $\left (T(t)\right )_{t\ge 0}$ generated by $A$ satisfies
\[
\| T(t)\|_{\mathbb{X}} \le e^{\omega t}
\quad \text{ for some fixed $\omega$ and all $t\ge 0$.}
\]
Otherwise, a new equivalent norm of $\mathbb{X}$ could be used for this purpose. When $A$ is the generator of an analytic semigroup, we may consider fractional powers $A^\alpha$ of the operator $A$ with $0<\alpha \le 1$. 
In this case, for simplicity, we will assume that $\sigma (A)$ is in the left half of the complex plane, so the Banach space $\mathbb{X}^\alpha:=D(A^\alpha)$ will be equipped with the norm $\| x\|_\alpha := \| A^\alpha x\|_{\mathbb{X}}$ for all $x\in \mathbb{X}^\alpha$. 
As shown in \cite{traweb2}, if $B:C([-r,0],\mathbb{X}^\alpha)\to\mathbb{X}$ is a bounded linear operator, then Eq.~\eqref{PFDE-00} generates a $\mathrm{C}_0$-semigroup in $\mathcal{C}_\alpha:= C([-r,0],\mathbb{X}^\alpha)$. 
We will use the formula for the generator of the solution semigroup generated by Eq.~\eqref{PFDE-00} from \cite[(3.3), p. 402]{traweb} or \cite[(4.3), p. 136]{traweb2} accordingly. 
We also use the notations and results in \cite{traweb3}.

\section{Yosida Distance and Linear Perturbation of Linear (Non-Delayed) Evolution Equations}
\label{Sect-YdPe}

We begin this section with the concept of Yosida distance between two linear operators. To this end, we recall the concept of Yosida approximation of a linear operator (see e.g. \cite{paz,yos}). Given operator $A$ in a Banach space $\mathbb{X}$ with $\rho(A) \supset [\omega, \infty)$, where $\omega$ is a given number, the \textit{Yosida approximation} $A_\lambda$ is defined as 
\[
A_\lambda := \lambda^2 R(\lambda, A)-\lambda I
\]
for sufficiently large $\lambda$.

\begin{definition}[Yosida distance]\label{def 1}
The \textit{Yosida distance} between two linear operators $A$ and $B$ satisfying $\rho(A) \supset [\omega, \infty)$ and $ \rho(B) \supset [\omega, \infty)$, where $\omega$ is a given number, is defined to be
\[
d_{Y}(A,B):=\limsup_{\mu\to +\infty} \left \| A_\mu-B_\mu\right \| .
\]
\end{definition}

The following theorem is the main result of this subsection on the roughness of exponential dichotomy under Yosida perturbation.
\begin{theorem}\label{the per}
Let $A$ be the generator of a $\mathrm{C}_0$-semigroup which admits an exponential dichotomy. Then, the $\mathrm{C}_0$-semigroup generated by a linear operator $B$ also has an exponential dichotomy, provided that $d_{Y}(A,B)$ is sufficiently small.
\end{theorem}
\begin{proof}
First, we assume that both semigroups generated by $A$ and $B$ are contraction $\mathrm{C}_0$-semigroups. Then, 
by \cite[Lemma~3.4]{paz}, $\left (e^{tA_\lambda} \right )_{t \geq 0}$ is a $\mathrm{C}_0$-semigroup of contractions. 

Let $C$ and $D$ be two bounded linear operators in a Banach space $\mathbb{X}$. We will estimate the growth of $e^{tC}-e^{tD}$. By the Variation-of-Constants Formula that is applied to the equation 
\[
x'(t)=Cx(t)+(D-C)x(t)
\]
and by setting $x(t)=e^{tD}x$ we have
\[
x(t)=e^{tC}x+\int\limits^t_0 e^{(t-s)C}(D-C)x(s)\mathrm{d}s.
\]
For each $t\ge 0$, we have
\begin{align*}
\left \|e^{tC}x-e^{tD}x\right \|
&\le
\int\limits^t_0 \left \| e^{(t-s)C}(D-C)e^{sD}x\right \|\mathrm{d}s
\\
&\le
t\| C-D\|  
\left \| e^{tC}\right \|
\left \| e^{tD}\right \|
\| x\|.
\end{align*}
Therefore,
\begin{equation*}
\left \| e^{tA_\lambda}-e^{tB_\lambda }\right \|
\le 
t 
\left \| A_\lambda-B_\lambda \right \|
\left \| e^{tA_\lambda }\right \| 
\left \|e^{tB_\lambda }\right \|.
\end{equation*}
Now we assume that $\left (T(t)\right )_{t\ge 0}$ and $\left (S(t)\right )_{t\ge 0}$ are the $\mathrm{C}_0$-semigroups generated by $A$ and $B$ that satisfy 
\[
\| T(t)\| \le Me^{\omega t},
\quad
\| S(t)\| \le Me^{\omega t}
\]
for certain positive numbers $M$ and $\omega$. 
As is well known, for the Yosida approximation $A_\lambda$ of the generator $A$ of a $\mathrm{C}_0$-semigroup $\left (T(t)\right )_{t \geq 0}$ that satistifies $\|T(t)\| \le Me^{\omega t}$ the following estimate of the growth is valid (see e.g. \cite[(5.25)]{paz})
\[
\left \| e^{tA_\lambda } \right \|
\le  Me^{2\omega t},
\quad
\left \| e^{tB_\lambda }\right \|
\le  Me^{2\omega t}.
\]
Hence, we have 
\[
\left \| e^{tA_\lambda}-e^{tB_\lambda }\right \|
\le 
t M^2
\left \| A_\lambda-B_\lambda \right \| e^{4 t\omega }.
\]
By \cite[Theorem 5.5]{paz}, for each $x\in \mathbb{X}$, we have
\begin{align*}
\|T(t)x-S(t)x\| 
&=
\lim_{\lambda \to\infty} 
\left \| e^{tA_\lambda }x-e^{tB_\lambda }x\right \|  
\\
& \le
tM^2e^{4\omega t} \limsup_{\lambda \to\infty } 
\left \| A_\lambda  - B_\lambda \right \|   
\\
&=
tM^2e^{4\omega t} d_{Y}(A,B).
\end{align*}
Finally, if $d_{Y}(A,B)$ is sufficiently small, $\|T(1)-S(1)\|$ is sufficiently small as well, and thus, $\left (S(t)\right )_{t\ge 0}$ has an exponential dichotomy.
\end{proof}

In the perturbation theory of linear equations of the form $u'(t)=Au(t)$ perturbed by an operator $C$, one usually studies the perturbed equation $u'(t)=(A+C)u(t)$. A standard assumption in the literature is the domain inclusion condition $D(A)\subset D(C)$; see, for example, \cite{cholei, dunsch, lun}. 
The main novelty of this work is that such a domain condition is not imposed as a fundamental assumption. Instead, as shown later, the notion of Yosida distance provides a framework that encompasses all classical perturbation settings under this condition. The following lemma clarifies this relationship.

\begin{lemma}\label{lem 1}
The following statements hold:
\begin{enumerate}
\item Let $A,\,B$ be the generators of contraction semigroups with $D(A)= D(B)$. Then, $A=B$, provided that 
$d_{Y}(A,B)=0$.
\item Let $A,\, B\in \mathcal{L}(\mathbb{X})$. Then
\[
d_{Y}(A,B)=\| A-B\|.
\]
\item If $A$ is the generator of a $\mathrm{C}_0$-semigroup $\left (T(t)\right )_{t \geq 0}$ such that $\| T(t)\| \le Me^{\omega t}$, and $C$ is a bounded operator, then $d_{Y}(A,A+C)$ is finite. Moreover,
\[
d_{Y}(A,A+C)\le M^2\|C\| .
\]
\end{enumerate}
\end{lemma}
\begin{proof}\
\begin{enumerate}
\item
From the semigroup theory 
(see \cite[Lemma 1.3.3]{paz}),
if $A$ is the generator of a contraction $\mathrm{C}_0$-semigroup, then
\[
\lim_{\mu\to+\infty} A_\mu x=Ax,
\quad
x\in D(A).
\]
Therefore, for $x\in D(A)=D(B)$, $Ax=Bx$. 

\item Firstly, we have
\begin{align*}
& R(\mu, A) - R(\mu, B)
\\
& =
(\mu - B)
R(\mu, B)
R(\mu, A)
-
R(\mu,B)
R(\mu,  A)
(\mu - A).
\end{align*}
Set
\[
C_{\mu}
:=
R(\mu, B)
R(\mu, A).
\]
Then
\begin{align*}
R(\mu, A) - R(\mu, B)
&= 
(\mu - B)
C_{\mu}
-
C_{\mu}
(\mu - A)
\\
&= 
\mu C_{\mu}
+
B C_{\mu}
-
C_{\mu} \mu
-
C_{\mu} A
\\
&= 
B C_{\mu} - C_{\mu} A.
\end{align*}
Thus,
\[
\mu^2
\|R(\mu, A) - R(\mu, B)\|
=
\mu^2
\left \|
B C_{\mu} - C_{\mu} A
\right \|.
\]
We will show that for a bounded linear operator $A$ in $\mathbb{X}$ the following is valid
\begin{equation}\label{7}
\lim_{\mu\to+\infty} R(\mu,A)=0.
\end{equation}
In fact, for sufficiently large $\mu$, say $\mu >\| A\|$, using the Neuman series, for large $\mu$, we have
\begin{align*}
\| R(\mu,A)\| 
= \frac{1}{\mu} \left\| R\left(1,\frac{1}{\mu}A\right) \right \|
&\le
\frac{1}{\mu} \left\| \sum_{n=0}^\infty \left( \frac{1}{\mu}A\right)^n \right\|  
\\
&\le
\frac{1}{\mu}
\frac{1}{1-  \frac{1}{\mu}\|A\|}.
\end{align*}
This proves \eqref{7}.
Next, by \eqref{7} and the identity
$(\mu-A)R(\mu,A)=I$
we have that
\begin{equation*}
\lim_{\mu\to+\infty} \mu R(\mu,A)=I,
\end{equation*}
so
\begin{equation*}
\lim_{\mu\to+\infty} \mu^2 C_\mu=I.
\end{equation*}
Finally, we have
\begin{align*}
d_{Y}(A,B)
& :=
\limsup_{\mu\to +\infty} 
\mu^2
\left \|
B C_{\mu} - C_{\mu} A\right \|
\\
& =
\|B-A\|,
\end{align*}
and this finishes the proof.

\item By a simple computation we have
\[
R(\mu, A+C)-R(\mu, A) = R(\mu, A+C)CR(\mu,A). 
\]
It is known (see e.g. \cite[Theorem 1.1, p.~76]{paz} that linear operator $A+C$ with $D(A+C)=D(A)$ generates a $\mathrm{C}_0$-semigroup $(S(t))_{t\ge 0}$ satisfying
\[
\| S(t)\| \le Me^{(\omega +M\|C\|)t}, 
\]
so by the Hille--Yosida Theorem
\begin{align*}
& \| R(\mu, A)\| \le \frac{M}{\mu -\omega},
\\
& \| R(\mu,  A+C)\| \le \frac{M}{\mu -(\omega+M\| C\|)}, 
\end{align*}
for certain positive constants $M$ and $\omega$. Therefore,
\begin{align*}
\limsup_{\mu \to \infty}  \mu^2 \| R(\mu, A)-R(\mu, A+C)\| 
& \le
\limsup_{\mu\to\infty } \frac{\mu^2M^2\| C\| }{(\mu -\omega)(\mu -(\omega +M\| C\|)} 
\\
&= M^2\| C\| <\infty. 
\end{align*}
\end{enumerate}

The proof is completed.
\end{proof}

By Lemma~\ref{lem 1}, the quantity $d_Y(\cdot,\cdot)$ extends the usual norm distance between bounded linear operators to a broader class of (possibly unbounded) operators for which the Yosida distance is finite. Examples of such operators are provided in \cite[Section 5]{builuomin}.

\section{Perturbation of Partial Functional Differential Equation $u'(t)=Au(t)+Bu_t$}
\label{Sect-PePFDE}
Consider partial functional differential equation of the form
\begin{equation}\label{315}
u'(t)=A_iu(t)+B_i u_t,
\end{equation}
where $A_i$ generates a $\mathrm{C}_0$-semigroup $\left (T_{i} (t)\right )_{t\ge 0}$ in $\mathbb{X}$, and $B_i$ is a bounded operator from $\mathcal{C}:=C([-1,0],\mathbb{X})$ to $\mathbb{X}$, for $i \in \{0, 1\}$, and 
\[
u_t (\theta ):=u(t+\theta ), \quad \theta \in [-1,0].
\]
By the standard theory of partial functional differential equation, this equation generates a solution semigroup $(\mathcal{T}_i(t))_{t\ge 0}$ that is strongly continuous in the phase space $\mathcal{C}$ with the generator $\mathcal{G}_i$ defined as
(see \cite[p. 402]{traweb})
\begin{align*}
\mathcal{G}_i\phi &= \phi' 
\quad \text{ on the domain} 
\\
D(\mathcal{G}_i)&= \{ \phi \in \mathcal{C}:  \phi(0) \in D(A_i),\, \phi'^- (0)=A_i\phi (0) +B_i\phi \} . 
\end{align*}
The investigation of Eq. \eqref{315} is, therefore, reduced to that of the evolution equation (without delay) 
\[
w'(t)=\mathcal{G}_i w(t)
\]
in the Banach space $\mathcal{C}$.

As shown below, the primary domain condition used in the perturbation theory as  discussed above is the inclusion of the domains of the operators $D(\mathcal{G}_0) \subset D(\mathcal{G}_1)$. We will see that, generally, this never happens to partial functional differential equations. However, the Yosida distance between $\mathcal{G}_0$ and $\mathcal{G}_1$ can be determined. That means, we can still study the perturbation of the system.

Below, for our convenience, we will drop the index $i$ in the above problem to get a general formula for the domain $D(\mathcal{G}_i)$. Suppose that $x\in D(A)$ is an arbitrary element that is fixed.  
Assume that $BD(A) \subset D(A)$. 
We define 
\[
\phi(t) := T(t)x +\int\limits^t_{0} T(t-\xi )Bx \mathrm{d}\xi .
\]
Then, 
\[
\phi (1) = T(1)x +\int\limits^1_{0}T(1-\xi )Bx\mathrm{d}\xi \in D(A).
\]
Moreover, by 
\cite[Theorem 2.4, p.~4]{paz},
\[
(T(t)x)' = AT(t)x 
\]
and 
\[
\frac{\mathrm{d}}{\mathrm{d}t} \left( \int\limits^t_{0} T(t-\xi )Bx\mathrm{d}\xi  \right) = Bx+\int\limits^t_{0} AT(t-\xi )Bx\mathrm{d}\xi.
\]
Therefore,
\begin{align*}
\phi'(1) &= A\left( T(1)x+\int\limits^1_{0} T(1-\xi )Bx\mathrm{d}\xi \right) +Bx \\
&= A\phi (1) +B\phi (0) .
\end{align*}
Consequently, if we set
\[
\varphi (t) = \phi (t+1), \quad t \in [-1,0],
\]
then the following is valid:
\begin{lemma}
Assume that for $i \in \{0, 1\}$, $B_iD(A_i)\subset D(A_i)$. Then, for any $x\in D(A_i)$ the function
\begin{equation}\label{1minh-varphi}
\varphi (t): =  T(t+1)x +\int\limits^{t+1}_{0} T(t+1-\xi )B_ix\mathrm{d}\xi,
\quad
t\in [-1,0] 
\end{equation}
is an element in $D(\mathcal{G}_i)$.
\end{lemma}
 
\begin{corollary}\label{cor 100}
Assume that $A_0=A_1=A$, and $B_i D(A)\subset D(A)$, for $i \in \{0, 1\}$. Then, $D(\mathcal{G}_0) \subset D(\mathcal{G}_1)$ if and only if 
$B_1=B_0$.
\end{corollary}
\begin{proof}
If we choose any $x\in D(A_0)$, then $\varphi (t)$ as defined in \eqref{1minh-varphi} is an element of $D(G_0)$. Hence, $\varphi \in D(G_1)$. This yields that
\begin{align*}
\varphi'(0) 
& = A\varphi (0) +B_0\varphi(-1)
\\
& = A\varphi (0) +B_1\varphi(-1).
\end{align*}
Therefore, $B_0\varphi (-1) =B_1\varphi (-1)$. In other words, $B_0x=B_1x$ for any given $x\in D(A)$. As $D(A)$ is dense everywhere in $\mathbb{X}$ and $B_0$ and $B_1$ are bounded linear operators, 
\[
B_0=B_1.
\]
This completes the proof.
\end{proof}

\begin{example}\label{exa 1}
Consider functional differential equations of the form
\begin{equation}\label{Exa1}
x'(t) = b_i x(t-1), \quad x(t)\in \mathbb{R},
\end{equation}
where $b_i\in \mathbb{R}$ is a given number for $i \in \{0, 1\}$. As shown in \cite{hal} this functional differential equation generates a $\mathrm{C}_0$-semigroup  in the phase space $\mathbb{X}:= C([-1,0],\mathbb{R})$  with the generator $\mathcal{G}_i$ defined as
\[
[\mathcal{G}_i\phi ](t)=\begin{cases}\phi'(t) 
& \text{ if }
t \in [-1,0),\\
b_i\phi (-1) 
& \text{ if }
t=0,
\end{cases}
\]
with the domain
\[
D(\mathcal{G}_i):= \left \{ \varphi \in C^1([-1,0],\mathbb{R}) 
:
\varphi'(0) =b_i \varphi (-1) \right \}.
\]
In this example, $A=0$ and $B_i=b_i$. 
Obviously, $B_i D(A) \subset D(A)=\mathbb{R}$. We have
\[
\varphi (t)=(1+(1+t)b_i)x
\quad \text{ with }
t\in [-1,0],
\]
where $x\in \mathbb{R}$ is any given number. We can see that $\varphi' (0) = b_ix =b_i\varphi (0)$, so $\varphi \in D(\mathcal{G}_i)$. Assuming that  $D(\mathcal{G}_0)\subset D(\mathcal{G}_1)$ as $x\in\mathbb{R}$ is arbitrary, we can choose a number $x=1$ to see that $b_0=b_1$.
\end{example}

Below we will demonstrate the usefulness of the concept of Yosida distance as it still works in this case when the persistent domain condition is not satisfied to apply the previously known perturbation theory. In fact we study the Yosida distance between the generators of the solution semigroups of partial functional differential equations of the form Eq.~\eqref{315}.

Consider partial functional differential equations of the form
\begin{equation}\label{pfdei}
u'(t)=A_iu(t)+B_iu_t, 
\end{equation}
where $i \in \{0, 1\}$, $A$ is the generator of a $\mathrm{C}_0$-semigroup $\left (T_i(t)\right )_{t\ge 0}$ satisfying $\| T_i(t)\| \le e^{\omega t}$ and $B_i\in \mathcal{L}(\mathbb{X})$ for a certain (fixed) positive constant $\omega$. We will estimate the Yosida distance between the generators $\mathcal{G}_0$ and $\mathcal{G}_1$ of the two solution semigroups generated $(\mathcal{T}_i(t))_{t\ge 0}$ by Eq.~\eqref{pfdei} in the phase space $\mathcal{C}$. For simplicity we will drop the index "$i$" in the following paragraph.

Below we will derive a general formula for the generators of the solution semigroups generated by Eq.~\eqref{pfdei}, where we will drop the index "$i$" as they are similar.
\begin{proposition}
The resolvent $R(\lambda, \mathcal{G})$ of the generator of the solution semigroup $\left (\mathcal S(t)\right )_{t \geq 0}$ in $\mathcal{C}$ is determined by the equation (in $\varphi$)
\begin{equation*}
\varphi (t)
= 
e^{\lambda t} \left[  R(\lambda , A)\left( B\varphi +\psi (0)\right)  \right]
- \int\limits_{0}^{t}e^{\lambda (t-\xi)}\psi(\xi)\mathrm{d}\xi ,
\end{equation*}
where $\varphi:= R(\lambda, \mathcal{G})\psi $ if $\psi$ is a given element of $\mathcal{C}$.
\end{proposition}

\begin{proof}
To find the resolvent $R(\lambda, \mathcal{G})$ for $\lambda > 0$, we will solve the following equation in $\mathcal{C}$
\begin{equation}\label{minh(2)-star}
\lambda \varphi - \mathcal{G}\varphi = \psi 
\quad\text{ for every given }
\psi \in \mathcal{C}.
\end{equation}
As $\varphi \in D(\mathcal{G})$, $\mathcal{G}\varphi(t) = \dot{\varphi}(t)$. From \eqref{minh(2)-star}, we have
\[
\lambda \varphi(t) - \dot{\varphi}(t) = \psi(t),
\]
equivalently
\[
\dot{\varphi}(t)= \lambda \varphi(t) - \psi(t), \quad\text{ for all } t \in [-1, 0].
\]
By the Variation-of-Constants Formula, we have
\begin{equation}\label{513}
\varphi (t) = e^{\lambda (t )} \varphi(0) - \int\limits_{0}^{t}e^{\lambda (t - \xi)}\psi(\xi)\mathrm{d}\xi, \quad\text{ for all } t \in [-1, 0].
\end{equation}
As $\varphi \in D(\mathcal{G})$, we have
\begin{equation*}
\lambda \varphi(0) - \psi (0)=\dot{\varphi}(0) = 
A\varphi(0) + B\varphi.
\end{equation*}
Hence, for large $\lambda$ we have
\[
\varphi (0)= R(\lambda , A)\left( B\varphi +\psi (0)\right),
\]
so
\eqref{513}, for $t\in [-1,0]$ we have
\begin{align*}
\left (R(\lambda, \mathcal{G})\psi\right ) (t)
& =
\varphi (t)
\\
& = 
e^{\lambda t} \left[  R(\lambda , A)\left( B\varphi +\psi (0)\right)  \right]
- \int\limits_{0}^{t}e^{\lambda (t-\xi)}\psi(\xi)\mathrm{d}\xi.
\end{align*}
The proof is complete.
\end{proof}

If we define the operator $\varphi  \mapsto \mathcal{F} \varphi (A,B)$ as 
\begin{align*} 
\mathcal{F}_\lambda (A,B)\varphi (t) &:= e^{\lambda t}  R(\lambda , A)B\varphi  ,
\end{align*}
and $\psi \mapsto \mathcal{J}_\lambda(A)\psi$ as
\[
\mathcal{J}_\lambda (A)\psi (t):= e^{\lambda t}  R(\lambda , A) \psi (0) - \int\limits_{0}^{t}e^{\lambda (t-\xi)}\psi(\xi)\mathrm{d}\xi,
\]
then, for each $\psi \in \mathcal{C}$, we have
\[
(I-\mathcal{F}_\lambda (A,B))\varphi = \mathcal{J}_\lambda (A)\psi .
\]
Note that as 
\begin{align}
\limsup_{\lambda\to +\infty} \| \mathcal{F}_\lambda\| &\le \limsup_{\lambda\to+\infty} \| R(\lambda ,A)\| \|B\| \notag \\
&= \limsup_{\lambda\to +\infty}\frac{\| B\| }{\lambda -\omega }\notag \\
&= 
0 .\label{331}
\end{align}
the operator $R(1, \mathcal{F}_\lambda )(A,B)$ exists, so
\begin{equation}\label{332}
\left (R(\lambda, \mathcal{G})\psi\right ) =\varphi = R(1, \mathcal{F}_\lambda )\mathcal{J}_\lambda \psi .
\end{equation}

\begin{lemma}\label{lem 39}
Let $\mathcal{G}_i$ be the generators of the solution semigroups generated by Eq.~\eqref{pfdei} in the phase space $\mathcal{C}:=C([-1,0], \mathbb{X})$. Assume further that
\[
d_Y(A_0,A_1)< \infty ,\quad 
\| B_i\| <\infty,
\qquad \text{ for }
i \in \{0, 1\}.
\]
Then,
\[
d_Y(\mathcal{G}_0, \mathcal{G}_1) \le 2 d_Y(B_0,B_1) +d_Y(A_0,A_1) .
\]
\end{lemma}
\begin{proof}
By \eqref{332} we have
\begin{align}
\left \|
R(\lambda, \mathcal{G}_0)
-
R(\lambda, \mathcal{G}_1)
\right \| 
&  \le 
\|  
R(1, \mathcal{F}_\lambda(A_0,B_0) )\mathcal{J}_\lambda (A_0)
-R(1, \mathcal{F}_\lambda(A_1,B_1))\mathcal{J}_\lambda (A_1)
 \| 
\notag \\
\hspace{1cm} & \le
\| R(1, \mathcal{F}_\lambda(A_0,B_0)) \mathcal{J}_\lambda (A_0) 
-R(1, \mathcal{F}_\lambda(A_1,B_1)) \mathcal{J}_\lambda (A_0)
\| 
\notag \\
\hspace{1cm} & 
\qquad +
\| R(1,\mathcal{F}_\lambda(A_1,B_1) )\mathcal{J}_\lambda (A_0) 
-R(1, \mathcal{F}_\lambda(A_1,B_1))\mathcal{J}_\lambda (A_1)
 \| 
\notag \\
\hspace{1cm} & \le
\| R(1, \mathcal{F}_\lambda(A_0,B_0))
-R(1, \mathcal{F}_\lambda(A_1,B_1))\| \cdot  \|\mathcal{J}_\lambda (A_0)
\| 
\notag \\
\hspace{1cm} & 
\qquad +
\| R(1, \mathcal{F}_\lambda(A_1,B_1))\| \cdot 
 \| \mathcal{J}_\lambda (A_1)
 -\mathcal{J}_\lambda (A_0)
 \|.
 \label{100}
\end{align}
We will estimate each term in \eqref{100}. For the second term, note that for sufficiently large $\lambda$, 
\begin{align*}
\| R(1, \mathcal{F}_\lambda(A_1,B_1))\| &\le \frac{1}{1-\| \mathcal{F}_\lambda(A_1,B_1)\|},
\\
\| \mathcal{J}_\lambda (A_1) -\mathcal{J}_\lambda (A_0)
\| 
& \le
\| R(\lambda, A_1)-R(\lambda ,A_0)\|,
\end{align*}
so by \eqref{331}
\[
\limsup_{\lambda\to +\infty} \lambda^2  \| R(1, \mathcal{F}_\lambda(A_1,B_1))\| \cdot \| \mathcal{J}_\lambda (A_1)
-
\mathcal{J}_\lambda (A_0)
\| 
\le  d_Y(A_0,A_1) .
\]
Now we estimate the first term. Using the identity for any $U,\, V\in \mathcal{L}(\mathbb{X})$ with sufficiently large $\lambda$
\[
R(\mu, U) - R(\mu, V)
=R(\mu,V)
R(\mu,  U)U  - VR(\mu, V)R(\mu, U)
\]
and notation
\[
C_\lambda := R(1,\mathcal{F}_\lambda(A_1,B_1))R(1,\mathcal{F}_\lambda(A_0,B_0)),
\]
we have
\begin{equation}
\label{336}
\| R(1, \mathcal{F}_\lambda(A_0,B_0)) -R(1, \mathcal{F}_\lambda(A_1,B_1))\|    
\le 
C_\lambda \mathcal{F}_\lambda(A_0,B_0) - \mathcal{F}_\lambda(A_1,B_1)C_\lambda.
\end{equation}
By \eqref{331} using the identity 
\[
R(1,U)-I = R(1,U)U
\]
we see that  
\begin{align*}
\| C_\lambda - I\| 
& \le
\| (R(1,\mathcal{F}_\lambda(A_1,B_1))-I)  R(1, \mathcal{F}_\lambda(A_0,B_0))\| 
\\
& \qquad 
+ \| R(1, \mathcal{F}_\lambda(A_0,B_0)) -I\|
\\
& =
\| R(1, \mathcal{F}_\lambda(A_1,B_1))\| \cdot \| \mathcal{F}_\lambda(A_1,B_1 )\| \cdot \| \mathcal{F}_\lambda(A_0,B_0))\|
\\
& \qquad 
+
\| R(1,\mathcal{F}_\lambda(A_0,B_0))\|
\cdot
\| \mathcal{F}_\lambda(A_0,B_0))\| .
\end{align*}
Consequently, since
\[
\| \mathcal{J}_\lambda (A_0)  \| \le \frac{2}{\lambda -\omega}
\]
we can easily show that 
\begin{align*}
&
\limsup_{\lambda\to +\infty} \lambda^2(C_\lambda - I)\mathcal{F}_\lambda(A_0,B_0))\mathcal{J}_\lambda (A_0) =0,
\\
&
\limsup_{\lambda\to +\infty} \lambda^2  \mathcal{F}_\lambda(A_1,B_1)(C_\lambda - I)\mathcal{J}_\lambda (A_0) =0.
\end{align*}
Hence, by \eqref{336}
\begin{align*}
& \limsup_{\lambda\to +\infty}  \lambda^2 \| \left(  R(1,\mathcal{F}_\lambda(A_0,B_0))
- R(1,\mathcal{F}_\lambda(A_1,B_1))  \right)
\mathcal{J}_\lambda (A_0) \|
\\
& \le \limsup_{\lambda\to +\infty}  \lambda^2 \left\| \mathcal{F}_\lambda(A_0,B_0) - \mathcal{F}_\lambda(A_1,B_1) \right\| \cdot \| \mathcal{J}_\lambda (A_0)\| \\
& \le  \limsup_{\lambda\to +\infty}  \lambda^2 \| R(\lambda ,A_0)B_0-R(\lambda ,A_0)B_1 +R(\lambda ,A_0)B_1-R(\lambda ,A_1)B_1\| \cdot \| \mathcal{J}_\lambda (A_0)\| \\
&\le  \limsup_{\lambda\to +\infty} \left( \frac{ \lambda^2 }{\lambda -\omega} \| B_0-B_1\| \cdot \frac{2}{\lambda -\omega} + \lambda^2 \|R(\lambda,A_0)-R(\lambda,A_1)\| \cdot \frac{2}{\lambda -\omega} \right) \\
&\le 2 \|B_0-B_1\| .
\end{align*}
Finally,
\[
\limsup_{\lambda\to +\infty}  \lambda^2  \| R(\lambda, \mathcal{G}_0) 
-
R(\lambda, \mathcal{G}_1)  \|  \le 2 d_Y(B_0,B_1)+d_Y(A_0,A_1),
\]
finishing the proof.
\end{proof}

\begin{definition}
Let $A$ generate a $\mathrm{C}_0$-semigroup in a Banach space $\mathbb{X}$ and $B \in \mathcal{L}(\mathbb{X})$. Then, the partial functional differential equation
\[
u'(t)=Au(t)+Bu_t,\quad t\ge 0,
\]
is said to have an \textit{exponential dichotomy} if the solution semigroup generated by this equation in $\mathcal{C}$ has an exponential dichotomy.
\end{definition}

As an immediate consequence of Lemma~\ref{lem 39} we have the following:
\begin{theorem}\label{the main1}
Let $A_i$ generate a $\mathrm{C}_0$-semigroup $\left (T_i(t)\right )_{t \geq 0}$, $i \in \{0, 1\}$, in a Banach space $\mathbb{X}$ that satisfies $\|T_i(t)\| \le Me^{\omega t}$, $t\ge 0$, and $B_0,\, B_1\in \mathcal{L}(\mathcal{C}, \mathbb{X})$. Assume that the partial functional differential equation 
\[
u'(t)=A_0u(t)+B_0u_t,\quad t\ge 0,
\]
has an exponential dichotomy. Then, the partial functional differential equation 
\[
u'(t)=A_1u(t)+B_1u_t,\quad t\ge 0,
\]
also has an exponential dichotomy if the Yosida distance $d_Y(A_0, A_1)$ and $ d_Y(B_0, B_1)$ are sufficiently small.
\end{theorem}

\begin{remark}
In the infinite-dimensional case Theorem~\ref{the main1} seems to be new as the exponential dichotomy of the perturbed equation has been known in the literature only when $A_0=A_1$. In the infinite-dimensional case, the "standard" Variation-of-Constants Formula is not available, making the study of perturbed equation much harder (see e.g.  \cite{adiezz, adiezzouh, hinmurnaimin}).
\end{remark}

\section{Perturbation in $\alpha$-Norm of Partial Functional Differential Equations}
\label{Sect-Pealphanorm}

In this section we consider partial functional differential equations of the form
\begin{equation}\label{PFDE2}
u'(t)=A_0u(t)+B_0u_t, \quad t\ge 0, 
\end{equation}
where $u(t)\in \mathbb{X}$, $\mathbb{X}$ is a Banach space, $A_0$ is the generator of an analytic semigroup in $\mathbb{X}$ such that $\Re \sigma (A) < 0$. When such an operator $A_0$ is given we consider the operator $(-A_0)^\alpha$ as defined in \cite{paz} with $0\le \alpha < 1$. We will denote 
\[
\mathbb{X}_0^\alpha :=D(A_0^\alpha),
\quad
C_{0,\alpha}:= C([-r,0]; \mathbb{X}_0^\alpha).
\]
We assume that $B_0:C_{0,\alpha} \to \mathbb{X}$ is a bounded linear operator. 
We will use  some standard results in the theory of partial functional differential equations in \cite{traweb}:
\begin{theorem}
Under the above mentioned assumptions on $A_0$ and $B_0$, Eq.~\eqref{PFDE2} generates a strongly semigroup of bounded linear operators $(S(t))_{t\ge 0}$ in $C_\alpha$ defined as
\[
S(t)\phi (\theta ):= u(t+\theta ), \quad \theta \in [-r,0],
\]
where $u(\cdot)$ is the unique solution of the Cauchy Problem associated with Eq.~\eqref{PFDE2} with initial condition 
\[
u_0=\phi.
\]
Moreover, the infinitesimal generator $\mathcal{G}$ of the semigroup $\left (S(t)\right )_{t\ge 0}$ is defined as
\begin{align*}
D(\mathcal{G})&= \{ \phi \in C_\alpha : \  \phi '\in C_\alpha , \phi(0)\in D(A), \phi '^-(0)=A_0\phi(0)+B_0\phi \},
\\
(\mathcal{G}\phi )(\theta )& =\phi' (\theta), \quad \theta \in [-r,0].
\end{align*}
\end{theorem}

As a perturbed equation of Eq.~\eqref{PFDE2} we consider partial functional differential equations
\begin{equation}\label{PFDEp}
u'(t)=A_1u(t)+B_1u_t, \quad t\ge 0 
\end{equation}
under the following assumptions:
\begin{assumption}\label{a1} Throughout this section with a given fixed $\alpha \in [0,1)$ we assume that
\begin{enumerate}[label=\textnormal{Assumption~(A\arabic*)},
ref=\textnormal{Assumption~(A\arabic*)},
leftmargin=*]
\item\label{ItemA1} 
$A_0$ is the generator of an analytic $\mathrm{C}_0$-semigroup $\left (T(t)\right )_{t\ge 0} $ such that there are positive constants $M$ and $ \omega$ for which the following is valid for $t\ge 0$
\begin{equation}\label{5.5}
\| T(t)\| \le e^{-\omega t}.
\end{equation}		
\item\label{ItemA2} 
$D(A_0) \subset D(A_1) $ and $A_1$ is $A_0$-bounded, that is, there are positive constants $\delta$ and $\gamma$ such that
\[
\| A_1 x\| \le \delta \| x\| + \gamma \| A_0x\|, \quad \text{ for all } x\in D(A_0).
\]

\item\label{ItemA3}
$
(A_1-A_0)A_0^{-\alpha} \in \mathcal{L} (\mathbb{X},\mathbb{X})
$.
\end{enumerate}

\begin{enumerate}[label=\textnormal{Assumption~(B\arabic*)},
ref=\textnormal{Assumption~(B\arabic*)},
leftmargin=*]

\item\label{ItemB1}
$
B_0 \in \mathcal{L}(C_{0,\alpha} ,\mathbb{X})$ and $ B_1 \in \mathcal{L}(C_{1,\alpha} ,\mathbb{X})
$.
\end{enumerate}
\end{assumption}
\begin{remark}
\
\begin{enumerate}
\item Under \ref{ItemA1}, $A^{-\alpha} \in \mathcal{L}(\mathbb{X},\mathbb{X})$ and the space $\mathbb{X}_{0}^\alpha := D(A_0^\alpha )$ with norm $
\| x\| _\alpha := \| A_0^\alpha x\|$ is well-defined (see \cite{hen,paz}). 
Moreover, for sufficiently small $\delta$ and $\gamma$ in \ref{ItemA2}, the operator $A_1$ is also the generator of an analytic $\mathrm{C}_0$-semigroup satisfying a similar inequality as \eqref{5.5}, so there is another $\alpha$ space $\mathbb{X}^\alpha _1 :=D(A_1^\alpha )$ with norm $\|x\|_{1,\alpha}:= A_1^\alpha$.

\item\label{XQR2} By \ref{ItemA3} the vector spaces $\mathbb{X}^\alpha_0$  and $X_1^\alpha$ are the same, and the norms $\| x\|_\alpha $ and $\| x\|_{1,\alpha}$ are equivalent (see \cite[Theorem~1.4.8,~p.~29]{hen}, see also Theorem~\ref{ThmHenryFower}).

\item As a consequence of the above remark \eqref{XQR2}, under \ref{ItemB1}, the operator $B_1$ is actually an element of $\mathcal{L}(C_{0,\alpha}, \mathbb{X})$.
\end{enumerate}
\end{remark}

\begin{lemma}\label{lem 40} 
Under Assumptions~\ref{a1}, for sufficiently small $\delta$, the partial functional differential equations
\[
u'(t)=A_iu(t)+B_iu_t, \quad t\ge 0 ,
\]
where $i \in \{0, 1\}$, generate $\mathrm{C}_0$-semigroups in $C_\alpha$ with the generators $\mathcal{G}_i$, $i \in \{0, 1\}$. Moreover,
if
\[
d_Y(B_0,B_1) <\infty,
\quad
d_Y(A_0,A_1) <\infty ,
\]
then,
\begin{equation}\label{alpha-99}
d_Y(\mathcal{G}_0, \mathcal{G}_1) \le 2 d_Y(B_0,B_1) +d_Y(A_0,A_1) <\infty .
\end{equation}
\end{lemma}
\begin{proof}
The proof can be done in the same way as in that of Lemma~\ref{lem 39}. 
In fact, under Assumptions~\ref{a1} with sufficiently small $\delta$, by \cite{traweb2}, Eq. \eqref{PFDEp} generates a solution semigroup $\left (S_1(t)\right )_{t\ge 0}$ in the same space $C_\alpha$ as does the unperturbed equation Eq.~\eqref{PFDE2}. 
If we denote their generators  by $\mathcal{G}_0$ and $\mathcal{G}_1$, respectively, then, by following the same lines as in the proof of Lemma~\ref{lem 39} we can prove \eqref{alpha-99}.
\end{proof}

\begin{theorem}\label{the main2}
Let the assumptions of Lemma \ref{lem 40} be made. Moreover, assume that the $\mathrm{C}_0$-semigroup with the generator $\mathcal{G}_0$ has an exponential dichotomy. Then, for $A_1$ and $B_1$ sufficiently close to $A_0$ and $B_0$ in the Yosida distance's sense, the semigroup with the generator $\mathcal{G}_1$ has also an exponential dichotomy.
\end{theorem}

\section{Examples}
\label{Sect-examples}

In this section, we provide illustrative examples that demonstrate the applicability of the main results obtained in the previous sections.

\begin{example}\label{XQExx1}
Consider 
reaction-diffusion equations with finite delay
\begin{equation}\label{PFDE}
\left\{
\begin{aligned}
\frac{\partial w(x, t)}{\partial t} 
&=
\frac{\partial^2 w(x, t)}{\partial x^2}
- a w(x,t) 
- b w(x, t-r), 
&&
0 \leq x \leq \pi,
&& t \geq 0,
\\
w(0, t) &= w(\pi, t) = 0,
&& 
&& t \geq 0,
\\
w(x, t) &= \varphi(t)(x),
&& 0 \leq x \leq \pi,
&& -r \leq t \leq 0,
\end{aligned}
\right .
\end{equation}
where $a$, $b$, and $r$ are given positive constants. 
\end{example}

We choose the Hilbert space $\mathbb{X} = L^2 [0, \pi]$ and consider the linear operator
\begin{align*}
A_T y &= \ddot{y} \quad \text{ on the domain }
\\
D(A_T)&=\{y\in\mathbb{X}:
\text{ $y$ and $\dot{y}$ are absolutely continuous, }
\ddot{y}\in\mathbb{X},\,
y(0) = y(\pi) = 0
\},
\end{align*}
and
\begin{equation*}
A_0:= A_T-aI, 
\qquad 
[B_0\phi](\theta)
:= \begin{cases} 
0 & \text{ if }  0<\theta \le r,\\
b\phi(-r) & \text{ if } \theta =0.
\end{cases}
\end{equation*}
Consider the evolution equation
\begin{equation}\label{PFDE_0}
u'(t)=A_0u(t)+B_0u_t
\end{equation}
in the Hilbert space $\mathbb{X}$. Note that $A_0$ generates an analytic compact semigroup in $\mathbb{X}$ (see \cite{traweb2}) and Eq.~\eqref{PFDE_0} generates a compact semigroup $(\mathcal{S}_0(t))_{t\ge 0}$ in $\mathcal{C}=C([-r,0], \mathbb{X})$. 
By the Spectral Mapping Theorem as $\left (\mathcal{S}_0(t)\right )_{t\ge 0}$ is compact, the exponential dichotomy of this semigroup can be determined by the fact that the spectrum of its generator does not intersect with the imaginary axis. 
In turn, the spectrum of its generator consists of only eigenvalues, so the spectrum of $A_T$ consists of eingenvalues
$
\sigma(A_T)=
\left \{-n^2: n \in\mathbb{N}\right \},
$
they are determined by the characteristic equation
\begin{equation}\label{cheq}
\lambda + a + be^{-\lambda r} = -n^2.
\end{equation}
Semigroup $\left (\mathcal{S}_0(t)\right )_{t\ge 0}$ has an exponential dichotomy if and only if Eq.~\eqref{cheq} has no root of the form $i\xi$, where $\xi \in\mathbb{R}$, for all $n=1,2, \dots$. For instance, $\left (\mathcal{S}_0(t)\right )_{t\ge 0}$  has an exponential dichotomy if either $b<\frac{1}{r}$ or $a+1>b$. 
So, below we will make the following assumption on Eq.~\eqref{PFDE}:
\begin{assumption}\label{A}
Either $b < \frac{1}{r}$ or $a+1 > b$.
\end{assumption}

As a perturbed version of Eq.~\eqref{XQExx1} (and hence of \eqref{PFDE_0}), we consider the following example.

\begin{example}
Consider the following reaction-diffusion equation with finite delay
\begin{equation}\label{PFDE-1}
\left\{
\begin{aligned}
\frac{\partial w(x, t)}{\partial t} 
&=
(1+\epsilon_1)\frac{\partial^2 w(x, t)}{\partial x^2} +\epsilon_2 \frac{\partial w(x, t)}{\partial x} 
\\
& \qquad 
- (a+\epsilon_3(x))w(x,t)
-b  w(x, t-r)
\\
& \qquad 
+
\int\limits^0_{-r}[\mathrm{d}\epsilon_4(\xi ) ]w(x, t+\xi )\mathrm{d}\xi,
&& 0\le t,
&& 0 \leq x \leq \pi,
\\
w(0, t)
&=
w(\pi, t) = 0,
&&
0\le t,
\\
w(x, t) 
&= 
\varphi(t)(x),
&& -r \leq t \leq 0,
&& 0 \leq x \leq \pi,
\end{aligned}
\right .
\end{equation}
where $\epsilon_1$ and $\epsilon_2$ are small constant, $\epsilon_3(\cdot)$ are smooth functions, $\epsilon_4(\cdot )$ is a left continuous function from $[-r,0]$ to $\mathcal{L}(\mathcal{C}, \mathbb{X})$ with bounded variation $\operatorname{Var} (\epsilon_4)$.

If we re-write the equation \eqref{PFDE-1} in the abstract form we have
\[
u'(t)=(A_0+A_1) u(t)+(B_0+B_1)u_t,
\]
where 
\begin{align*}
A_1y &:= \epsilon_1\ddot y + \epsilon_2  \dot y -\epsilon_3(\cdot )y, \quad y\in D(A_T),
\\
B_1 v &:= \int\limits^0_{-r}[\mathrm{d} \epsilon_4(\xi ) ]v(t+\xi )\mathrm{d}\xi, \quad v \in \mathcal{C}.
\end{align*}
\begin{claim}\label{cl1}
Linear operator $A_1$ is an $A_0$-bounded operator. 
More precisely, for each $y\in D(A_T)$,
\begin{equation}\label{6.9}
\| A_1y\| \le \left (\epsilon_1 +\frac{\epsilon_2}{2}\right ) A_0y 
+
\sup_{s\in [0,\pi]} |\epsilon_3(s)| \| y\|.
\end{equation}
\end{claim}
 \begin{proof}
By \cite[Example 2.2, p. 169--170]{engnag}, if we choose $\epsilon =\frac{1}{2}$, then
\[
\| \dot y \|_2 \le 18 \| y\|_2 +\frac{1}{2} \| \ddot y\|_2.
\]
Therefore, \eqref{6.9} follows.
\end{proof}
\end{example}
\begin{corollary}
Under Assumption~\ref{A}, for sufficiently small $\epsilon_1$, $\epsilon_2$, $\sup_{s\in [0,\pi]} |\epsilon_3(s)|$, and $\operatorname{Var}(\epsilon_4)$,
the solution semigroup generated by Eq.~\eqref{PFDE-1} in $\mathcal{C}$ has an exponential dichotomy.
\end{corollary}
\begin{proof}
The claim of the corollary follows directly from Claim~\ref{cl1}, Lemma~\ref{lem 1}, Lemma~\ref{lem 39} and Theorem~\ref{the per}.
\end{proof}

Finally, as an example of a perturbed equation in $\mathcal{C}_\alpha$ of Eq.~\eqref{PFDE_0} with $\alpha = \tfrac{1}{2}$, we consider the following example.
\begin{example}\label{Ex1}
Consider the following reaction-diffusion equation with finite delay:
\begin{equation}\label{PFDE-2}
\left\{
\begin{aligned}
\frac{\partial w(x, t)}{\partial t} 
&=
(1+\epsilon_1)\frac{\partial^2 w(x, t)}{\partial x^2}  
\\
& \qquad
- (a+\epsilon_3(x))w(x,t)  
- b  w(x, t-r)
\\
& \qquad
+
\epsilon_5 \frac{\partial w(x, t-r)}{\partial x}
\\
& \qquad
+ 
\int\limits^0_{-r} [\mathrm{d}\epsilon_4(\xi )]w(x, t+\xi ),
&& 0\le t,
&& 0 \leq x \leq \pi,
\\
w(0, t) &= w(\pi, t) = 0,
&& 0\le t ,
\\
w(x, t) &= \varphi(t)(x), 
&& -r \leq t \leq 0,
&& 0 \leq x \leq \pi,
\end{aligned}
\right .
\end{equation}
where $\epsilon_1$, $\epsilon_3$, and $\epsilon_5$ are small constants, and $\epsilon_4(\cdot)$ is a left continuous function with bounded variation on $[-r,0]$ with values in $\mathcal{L} (\mathbb{X})$.

We will re-write Eq.~\eqref{PFDE-2} in an abstract form as in \cite{traweb2}
\[
u'(t)=(A_0+A_2)u(t)+(B_0+B_2)u_t,
\]
where 
\begin{align*}
A_2u &:= \epsilon_ 1 \ddot 
u -\epsilon_3(\cdot ) u, \quad \text{ for all } u\in D(A_2):=D(A_T),
\\
B_2v &:=  \epsilon_5 \dot v(-r)+\int\limits^0_{-r} [\mathrm{d}\epsilon_4(\xi )]v(\xi ), \quad \text{ for all } v\in C_\alpha =C([-r,0], \mathbb{X}^\alpha_2 ),
\end{align*}
where $\mathbb{X}_2^\alpha$ is defined by $(A_0+A_2)$ as the generator of an analytic semigroup,

Here we denote by $(\mathbb{X}_1^\alpha, \| \cdot \|_{1,\alpha})$ (by $(\mathbb{X}_2^\alpha, \| \cdot \|_{1,\alpha})$, respectively) the Banach space determined by $A_0$ (by $(A_0+A_2)$, respectively) as the generator of an analytic semigroup in $\mathbb{X}$. As shown in \cite[Example 5.1]{traweb2}, $B_0+B_2$ defines a bounded linear operator from $C ([-r,0],\mathbb{X}_2^{\alpha})$ to $\mathbb{X}$. Moreover, by Theorem~\ref{ThmHenryFower}, the spaces $\mathbb{X}^\alpha_1=\mathbb{X}^\alpha_2$ with equivalent norms, so we can identify $\mathbb{X}^\alpha_2$ as $\mathbb{X}^\alpha_1$ and $B_0+B_2$ is a bounded linear operator from  $C([-r,0], \mathbb{X}_1^{\alpha})$ to $\mathbb{X}$. That yields that the Yosida distance $d_Y(B_0, B_0+B_2)$ is defined and finite. We will show that this Yosida distance can be made as small as we like if $\epsilon_5$ and $\epsilon_6$ are taken sufficiently small. In fact, we have
\begin{claim}\label{clai 3}
Under the above notations, 
\[
d_Y(B_0,B_0+B_2) \le \epsilon_5+\operatorname{Var}(\epsilon_4).
\]
\end{claim}
\begin{proof}
By \cite[Example 5.1]{traweb2}, it is shown that
\begin{align*}
& \| \phi(-r)\| _{\mathbb{X}} \le	\| \phi \|_{1,\alpha} ,
\\
& \| \phi'(-r)\| _{\mathbb{X}} \le	\| \phi \|_{1,\alpha}  .
\end{align*}
Therefore, $B_0$ and $B_2$ are bounded operators from  $C([-r,0],\mathbb{X}_1^{\alpha})$ to $\mathbb{X}$, so we have
\[
d_Y(B_0,B_0+B_2)=\| B_2\|_{1,\alpha}.
\]
Therefore, for $\phi\in C_{\alpha}$
\begin{align*}
\| B_2\phi\|_{\mathbb{X}}
&=
\left \| \epsilon_5 \dot \phi(-r)+\int\limits^0_{-r} [\mathrm{d}\epsilon_4(\xi) \phi(\xi )\right \|_{\mathbb{X}}\\ 
&\le 
\epsilon_5\left \|\dot \phi\right \|_{\mathbb{X}}+\operatorname{Var}(\epsilon_4 ) \| \phi\|_{\mathbb{X}}\\
&\le 
\epsilon_5\|\phi\|_{1,\alpha}+\operatorname{Var}(\epsilon_4 ) \| \phi\|_{1,\alpha}.
\end{align*}
Consequently,
\[
d_Y(B_0,B_0+B_2)\le \epsilon_5\|\phi\|_{1,\alpha}+\operatorname{Var}(\epsilon_4 ) \| \phi\|_{1,\alpha}.
\]
This finishes the proof.
\end{proof}

\begin{corollary}
Under Assumption~\ref{A} for sufficiently small $\epsilon_1$, $\epsilon_3$, $\operatorname{Var}(\epsilon_4)$, and $ \epsilon_5$, the solution semigroup generated by Eq.~\eqref{PFDE-2} in $C([-r,0],\mathbb{X}^\alpha_2)$ has an exponential dichotomy.
\end{corollary}
\begin{proof}
For sufficiently small $\epsilon_1$ and $\epsilon_3$ as the spaces $(\mathbb{X}^\alpha_2,\| \cdot \|_{2,\alpha})$ and $(\mathbb{X}^\alpha_1,\| \cdot \|_{1,\alpha})$ are equivalent.
We are going to show that Eq.~\eqref{PFDE2} generates a solution semigroup $(\mathcal S_2(t))_{t \geq 0}$ in $C([-r,0],\mathbb{X}^\alpha _1)$ that has an exponential dichotomy. Then, by the equivalence of the norms this yields that this solution semigroup has an exponential dichotomy in $C([-r,0],\mathbb{X}^\alpha _2)$ as well. For sufficiently small $\epsilon_1$ and $\epsilon_3$ by Lemma \ref{lem 1} the Yosida distance $d_Y(A_0,A_0+A_2)$ could be made as small as we like. Next, the Yosida distance $d_Y(B_0,B_0+B_2)$ could be made as small as we like by taking $\epsilon_5$ and $\epsilon_6$ sufficiently small. Finally, for the Yosida distance $d_Y(\mathcal G_0, \mathcal  G_2)$, where $\mathcal G_2$ is the generator of the solution semgroup  $(\mathcal S_2(t))_{t \geq 0}$, could be made as small as we like by taking $\epsilon_1$, $\epsilon_3$, $\operatorname{Var}(\epsilon_4)$, and $\epsilon_5$ sufficiently small. By Theorem~\ref{the per} this yields that for sufficiently small $\epsilon_1$, $\epsilon_3$, $\operatorname{Var}(\epsilon_4)$, and $\epsilon_5$ the solution semigroup $(\mathcal{S}_2(t))_{t \geq 0}$ has an exponential dichotomy in $C([-r,0], \mathbb{X}^\alpha _1)$. Since  $(\mathbb{X}^\alpha_2,\| \cdot \|_{2,\alpha})$ and $(\mathbb{X}^\alpha_1,\| \cdot \|_{1,\alpha})$ are equivalent this implies that $(\mathcal{S}_2(t))_{t \geq 0}$ has an exponential dichotomy in $C([-r,0],\mathbb{X}^\alpha _2)$.
\end{proof}
\end{example}

\subsection*{Acknowledgements}
The first-named author (X.-Q. Bui) thanks the Vietnam Institute for Advanced Study in Mathematics (VIASM) for financial support and a stimulating working environment.

\subsection*{Data availability statement}
No data was used for the research described in the manuscript.

\subsection*{Declarations}
\paragraph*{Conflict of interest}
The authors declare that there is no conflict of interest.


\end{document}